\documentclass[12pt, twoside]{article}
\usepackage{times}
\usepackage{bm}
\usepackage{graphics}
\usepackage{graphicx}
\usepackage{amsmath}
\usepackage{latexsym}
\usepackage{amssymb,mathrsfs}
\usepackage{subfigure}
\usepackage[flushmargin]{footmisc}

\newtheorem{thm}{Theorem}[section]
\newtheorem{lem}{Lemma}[section]
\newtheorem{prop}{Proposition}[section]

\newtheorem{cond}{Condition}[section]{\bf}{\rm}

{\bf}{\rm}
\newtheorem{assumpt}{Assumption}[section]{\bf}{\rm}

\newtheorem{rem}{Remark}[section]{\itshape}{\rmfamily}

\newenvironment{proof}{\noindent{\it Proof.~~}}{\qed\medskip}
\newcommand{\qed}{\hspace*{\fill}$\Box$}
\makeatletter
\def\eqnarray{\stepcounter{equation}\let\@currentlabel=\theequation
\global\@eqnswtrue
\global\@eqcnt\z@\tabskip\@centering\let\\=\@eqncr
$$\halign to \displaywidth\bgroup\@eqnsel\hskip\@centering
  $\displaystyle\tabskip\z@{##}$&\global\@eqcnt\@ne 
  \hfil$\;{##}\;$\hfil
  &\global\@eqcnt\tw@ $\displaystyle\tabskip\z@{##}$\hfil 
   \tabskip\@centering&\llap{##}\tabskip\z@\cr}
\makeatother
\makeatletter
    \renewcommand{\theequation}{%
    \thesection.\arabic{equation}}
    \@addtoreset{equation}{section}
  \makeatother

\newcommand{\dm}{\displaystyle}
\newcommand{\vc}{\bm}

\def\pasub#1{\hspace{0.1em}{}_{(#1)}\hspace{-0.05em}}
\def\sqsub#1{\hspace{0.1em}{}_{[#1]}\hspace{-0.05em}}
\def\lsub#1{\hspace{0.1em}{}_{#1}\hspace{-0.05em}}
\newcommand{\ol}{\overline}

\newcommand{\wt}{\widetilde}

\newcommand{\EE}{\mathsf{E}}
\newcommand{\PP}{\mathsf{P}}

\newcommand{\bbA}{\mathbb{A}}
\newcommand{\bbC}{\mathbb{C}}

\newcommand{\bbE}{\mathbb{E}}

\newcommand{\bbK}{\mathbb{K}}

\newcommand{\bbN}{\mathbb{N}}
\newcommand{\bbR}{\mathbb{R}}
\newcommand{\bbS}{\mathbb{S}}
\newcommand{\bbU}{\mathbb{U}}

\newcommand{\bbZ}{\mathbb{Z}}

\newcommand{\rmd}{{\rm d}}
\newcommand{\rme}{{\rm e}}

\renewcommand{\labelenumi}{(\alph{enumi})}
\newcommand{\dd}[1]{\if#11 1\!\!1 
\else {\if#1C I\!\!\!C
\else {\if#1G I\!\!\!G 
\else {\if#1J J\!\!\!J 
\else {\if#1S S\!\!\!S
\else {\if#1Z Z\!\!\!Z
\else {\if#1Q O\!\!\!\!Q
\else I\!\!#1
\fi}
\fi}
\fi}
\fi} 
\fi} 
\fi} 
\fi} 

\newcommand{\eqn}[1]{(\ref{eqn:#1})}
\newcommand{\lemt}[1]{Lemma~\ref{lem:#1}}

\newcommand{\asst}[1]{Assumption~\ref{ass:#1}}

\newcommand{\assn}[1]{\ref{ass:#1}}

\newcommand{\br}[1]{\langle #1 \rangle}

  \newcommand{\rmn}[1]{\if#11I\else {\if#12I\hspace{-0.12ex}I\hspace{-0.6ex}\else {\if #13I\hspace{-0.16ex}I\hspace{-0.16ex}I\hspace{-1.1ex}\else {\if #14I\hspace{-0.16ex}V\hspace{-1.6ex} \else V\hspace{-1.0ex}\fi}\fi} \fi} \fi}

\makeatother

\begin{document}\thispagestyle{plain} 

\hfill

{\Large{\bf
\begin{center}
Simple error bounds for the QBD approximation of a special class of two dimensional reflecting random walks
\end{center}
}
}

\begin{center}
{
Hiroyuki Masuyama%
${}^{1}
\footnote[2]{The research of the first author was supported by JSPS KAKENHI Grant No.~15K00034.}$, 
Yutaka Sakuma
${}^{2}
\footnote[3]{The research of the second author was supported by JSPS KAKENHI Grant No.~16K21704.
}$ 
and Masahiro Kobayashi${}^{3}$ 
}

\medskip

{\footnotesize
${}^1$%
Graduate School of Informatics, Kyoto University, Kyoto 606-8501, Japan

${}^2$%
School of Electrical and Computer Engineering, National Defense Academy, Kanagawa 239-8686, Japan

${}^3$%
School of Science, Tokai University Kanagawa 259-1292, Japan
}
\bigskip
\medskip

{\small
\textbf{Abstract}

\medskip

\begin{tabular}{p{0.85\textwidth}}
This paper considers the QBD approximation of a special class of
two-dimensional reflecting random walks (2D-RRWs). A typical example
of the 2D-RRWs is a two-node Jackson network with cooperative
servers. The main contribution of this paper is to provide simple
upper bounds for the relative absolute difference between the
time-averaged functionals of the original 2D-RRW and its QBD
approximation.
\end{tabular}
}
\end{center}

\begin{center}
\begin{tabular}{p{0.90\textwidth}}
{\small

\medskip

}
\end{tabular}

\end{center}

\section{Introduction}
\label{sec:introduction}

This paper considers the stationary distribution of a discrete-time
two dimensional reflecting random walk (2D-RRW) on the lattice quarter
plane. Such 2D-RRWs appear as the joint queue length processes of
two-node queueing systems and two-waiting-line queueing systems. Thus,
we can evaluate the long-run performance of these queueing systems
once we can obtain the stationary distributions of the corresponding
2D-RRWs.

Unfortunately, it is, in general, difficult to obtain a closed-form
expression of the stationary distribution of 2D-RRWs. This is
primarily why the tail asymptotics of 2D-RRWs and their
generalizations have been extensively studied (see
\cite{Boro01,Guil11,Koba13,Koba14,LiHui11,LiHui13,Miya09,Miya11,Miya15,Ozaw13}
and the references therein). These studies focus on identifying the
decay rate of the stationary distribution, though they make a limited
contribution to the performance evaluation of the queueing systems
mentioned above.

On the other hand, for some special 2D-RRWs, the product-form solution
\cite{Lato14}, the mixed-geometric-form solution \cite{Chen15} and the
partially geometric solution \cite{Koba15} of the stationary
distribution are derived.  Although these solutions are tractable and
useful, they require restrictive conditions. Thus, the literature
\cite{Chen15,Gose14,Lato14} discussed the approximations such that a
2D-RRW is perturbed to another 2D-RRW having the stationary
distribution in product form or mixed-geometric form. The studies
\cite{Chen15,Gose14} also proposed the linear programming method for
establishing error bounds for the linear time-averaged functionals of
2D-RRWs, such as the mean value of the stationary distribution. This
linear programming method produces an error bound as a solution of the
linear program, and therefore the obtained bound is not explicit.

In fact, it is suggested in \cite{Lato14} that the QBD approximation
of 2D-RRWs yields very exact results when the truncation point of the
coordinate is sufficiently large. Note that the QBD approximation is
such that a 2D-RRW is reduced to a quasi-birth-and-death process (QBD)
by truncating one of the two coordinates of the original
2D-RRW. Motivated by this suggestion, we focus on developing computable error
bounds for the QBD approximation of 2D-RRWs.

In this paper, we assume some conditions on the mean drifts of the
2D-RRW. Under the conditions, we establish a geometric
(Foster-Lyapunov) drift condition on the transition probability matrix
of the 2D-RRW. Using the geometric drift condition and the upper bound for the deviation matrix (see, e.g., \cite{Cool02}) of the 2D-RRW, we develop a relative error bound for the approximate time-averaged functional obtained by the QBD approximation.  The error bound includes
the stationary distribution of the QBD approximation, which can be
readily computed by matrix analytic methods (see, e.g.,
\cite{Lato99}). Thus, the error bound is also computable. In addition,
from the error bound, we derive another bound by removing the
stationary distribution of the QBD approximation. The second error
bound is weaker but simpler than the first one.

\section{Preliminaries}
\label{sec:reflecting random walk}

In this section, we first introduce the two dimensional reflecting
random walk (2D-RRW) and its stability condition. We then describe the  technical conditions used to develop our error bounds, while discussing the
moment generating functions of the increments of the 2D-RRW. Finally, we
establish the geometric drift condition on the 2D-RRW.

\subsection{Two dimensional reflecting random walk}

Let $\{\vc{Z}(\ell) :=(Z_1(\ell), Z_2(\ell));\ell \in \bbZ_+\}$ denote
a two-dimensional Markov chain with state space $\bbS:=\bbZ_+^2=\bbZ_+
\times \bbZ_+$, where $\bbZ_+ = \{0,1,2,\dots\}$. For
$\vc{n}:=(n_1,n_2) \in \bbS$ and $\vc{m}:=(m_1,m_2) \in \bbS$, let
$p(\vc{n};\vc{m})$ denote
\begin{equation}
p(\vc{n};\vc{m}) = \PP(\vc{Z}(\ell+1) = \vc{m} \mid \vc{Z}(\ell) = \vc{n}),
\qquad \ell \in \bbZ_+.
\label{defn-p(n;m)}
\end{equation}

To describe the behavior of $\{\vc{Z}(\ell)\}$, we introduce some
definitions and notation. Let $\bbR = (-\infty,\infty)$, 
$\bbZ=\{0,\pm1,\pm2,\dots\}$ and $\bbN=\{1,2,\dots\}$. Furthermore, let $\bbE = \{1,2\}$ and $2^{\bbE}$ denote the power set of $\bbE$, i.e., $2^{\bbE} =
\{\emptyset, \{1\},\{2\},\bbE\}$. We then define $\bbS^{\bbA}$'s, $\bbA \in 2^{\bbE}$, as disjoint subsets of $\bbZ_+^2$ such that
\[
\bbS^{\emptyset} =\{(0,0)\}, \quad
\bbS^{\{1\}} = \bbN \times \{0\}, \quad
\bbS^{\{2\}} = \{0\} \times \bbN, \quad
\bbS^{\bbE} = \bbN^2.
\]
Clearly, $\cup_{\bbA \in 2^{\bbE}} \bbS^{\bbA} =
\bbS$. We refer to $\bbS^{\bbE}$ and $\bbS
\setminus \bbS^{\bbE}$ as the {\it interior} and {\it boundary},
respectively, of the state space of $\bbS$. We also refer to
$\bbS^{\emptyset}$, $\bbS^{\{1\}}$ and $\bbS^{\{2\}}$ as the {\it
  boundary faces} of the state space $\bbS$.

For $\bbA \in 2^{\bbE}$, let $\vc{X}^{\bbA} :=
(X_{1}^{\bbA},X_{2}^{\bbA})$ denote a random vector in
$\bbZ^2$ such that $\PP(\vc{X}^{\bbA} \in
\bbU^{\bbA}) = 1$, where
\begin{align*}
\bbU^{\emptyset} &= \{0,1\} \times \{0,1\}, &
\bbU^{\{1\}} &= \{0,\pm 1\} \times \{0,1\}, 
\\
\bbU^{\{2\}} &= \{0,1\} \times \{0,\pm 1\}, &
\bbU^{\bbE} &= \{0,\pm 1\} \times \{0,\pm 1\}.
\end{align*}
Furthermore, let $\vc{X}^{\bbA}(\ell):=(X_{1}^{\bbA}(\ell),X_{2}^{\bbA}(\ell))$'s,
$\ell \in \bbZ_+$, denote independent copies of $\vc{X}^{\bbA}$. Thus, for all $\ell \in \bbZ_+$,
\begin{equation}
\PP(\vc{X}^{\bbA}(\ell) = \vc{m}) 
= \PP(\vc{X}^{\bbA}= \vc{m}),
\qquad \vc{m} \in \bbZ^2.
\label{defn-X_l^{(k)}}
\end{equation}
We now assume that 
\begin{equation}
\vc{Z}(\ell + 1)
= \vc{Z}(\ell) + \sum_{\bbA \in 2^{\bbE}} 
\vc{X}^{\bbA}(\ell) I(\vc{Z}(\ell) \in \bbS^{\bbA}),
\qquad \ell \in \bbZ_+,
\label{defn-Z_l}
\end{equation}
where $I(\,\cdot\,)$ denotes the indicator function of the event
in the parentheses. We then define $p^{\bbA}$, $\bbA \in 2^{\bbE}$, as the distribution
function of $\vc{X}^{\bbA}$ such that $\sum_{\vc{m} \in \bbU^{\bbA}}
p^{\bbA}(\vc{m}) = 1$ and
\begin{eqnarray}
p^{\bbA}(\vc{m}) &=& \PP(\vc{X}^{\bbA} = \vc{m}),
\qquad \vc{m} \in \bbZ^2.
\label{defn-X^{(k)}}
\end{eqnarray}
It follows from (\ref{defn-X_l^{(k)}}), (\ref{defn-Z_l}) and
(\ref{defn-X^{(k)}}) that, for $\vc{n} \in \bbS^{\bbA}$ and $\bbA \in 2^{\bbE}$,%
\begin{eqnarray}
\PP(\vc{Z}(\ell + 1) = \vc{n}+\vc{m} \mid \vc{Z}(\ell) = \vc{n})
&=& p^{\bbA}(\vc{m}),\qquad \vc{m} \in \bbZ^2.
\label{eqn:transition-law}
\end{eqnarray}
It also follows from (\ref{defn-p(n;m)}),
(\ref{eqn:transition-law}) and $\sum_{\vc{m} \in \bbU^{\bbA}}
p^{\bbA}(\vc{m}) = 1$ that
\begin{equation}
p(\vc{n};\vc{m})
= p^{\bbA}(\vc{m}-\vc{n}),\quad 
\vc{n} \in \bbS^{\bbA},\ \vc{m}-\vc{n} \in \bbU^{\bbA},\ 
\bbA \in 2^{\bbE}.
\label{eqn-p(n,m)}
\end{equation}

In what follows, we refer to $\{\vc{Z}(\ell)\}$ described above as a
{\it two dimensional reflecting random walk} ({\it 2D-RRW}). We also
refer to $p^{\bbA}$ and $\{\vc{X}^{\bbA}(\ell)\}$ as the {\it
  transition law} and {\it increment}, respectively, in $\bbS^{\bbA}$.

\subsection{Stability condition}

In this subsection, we provide the summary of the known results on the
stability condition (ergodic condition) of the 2D-RRW $\{\vc{Z}(\ell)\}$.

For $\bbA \in 2_+^{\bbE}:=\{\{1\},\{2\},\bbE\}$, let
\begin{equation}
\vc{\mu}^{\bbA} 
:= (\mu_1^{\bbA},\mu_2^{\bbA}) 
= (\EE[X^{\bbA}_1],\EE[X^{\bbA}_2]).
\label{defn-mu}
\end{equation}
It follows from (\ref{defn-Z_l}) that $\vc{\mu}^{\bbA}$ is the vector
of the mean increments of the 2D-RRW $\{\vc{Z}(\ell)\}$ in
$\bbS^{\bbA}$. Thus, we call $\vc{\mu}^{\bbA}$ the {\it mean drift} in
$\bbS^{\bbA}$. By definition, $X_2^{\{1\}} \ge 0$ and $X_1^{\{2\}} \ge
0$ with probability one (w.p.1), which leads to
\begin{equation}
\mu_2^{\{1\}} \ge 0, \qquad \mu_1^{\{2\}} \ge 0.
\label{ineqn-mu}
\end{equation}
In addition, for any two vectors $\vc{x}=(x_1,x_2)$ and
$\vc{y}=(y_1,y_2)$ in $\bbR^2$, let 
\[
\vc{x} \wedge \vc{y} = x_1 y_2 - x_2 y_1.
\]
Note here that $\vc{x} \wedge \vc{y}$ is equivalent to the third
element of the cross product of two vectors $(\vc{x},0)$ and
$(\vc{y},0)$ in $\bbR^3$.  Therefore, $\vc{x} \wedge \vc{y} > 0$
(resp.~$\vc{x} \wedge \vc{y} < 0$) if and only if the direction angle
of vector $\vc{y}$ from vector $\vc{x}$ is in the range $(0,\pi)$
(resp.~$(-\pi,0)$), where the positive direction is counterclockwise.

In the rest of this paper, we assume that the 2D-RRW
$\{\vc{Z}(\ell)\}$ is irreducible and aperiodic. We also assume the
following stability condition of the 2D-RRW $\{\vc{Z}(\ell)\}$.
\begin{assumpt}[Stability condition]\label{ass:stability} 
Either of the following is satisfied:
\begin{enumerate}
\item $\mu_{1}^{\bbE} < 0$, $\mu_{2}^{\bbE} < 0$,
  $\vc{\mu}^{\bbE} \wedge \vc{\mu}^{\{1\}} < 0$ and
  $\vc{\mu}^{\bbE} \wedge \vc{\mu}^{\{2\}} > 0$.
\item $\mu_{1}^{\bbE} \ge 0$, $\mu_{2}^{\bbE} < 0$ and 
$\vc{\mu}^{\bbE} \wedge \vc{\mu}^{\{1\}} < 0$. 
In addition, $\mu^{\{2\}}_{2} < 0$ if $\mu^{\{2\}}_{1} = 0$.
\item $\mu_{1}^{\bbE} < 0$, $\mu_{2}^{\bbE} \ge 0$ and 
$\vc{\mu}^{\bbE} \wedge \vc{\mu}^{\{2\}} > 0$.  
In addition, $\mu^{\{1\}}_{1} < 0$ if $\mu^{\{1\}}_{2} = 0$.
\end{enumerate}
\end{assumpt}

It is known (see, e.g., \cite{Koba13}) that if
Assumption~\ref{ass:stability} holds then the 2D-RRW
$\{\vc{Z}(\ell)\}$ has the unique stationary distribution, denoted by
$\vc{\pi}:=(\pi(n_{1},n_{2}))_{(n_1,n_2) \in \bbS}$. The geometric
interpretation of this stability condition is summarized in
Figs.~\ref{fig:stability-a}, \ref{fig:stability-b} and \ref{fig:stability-c}.
\begin{figure}[h]
\centering
\subfigure[Case (a)]{
\includegraphics[scale=0.34]{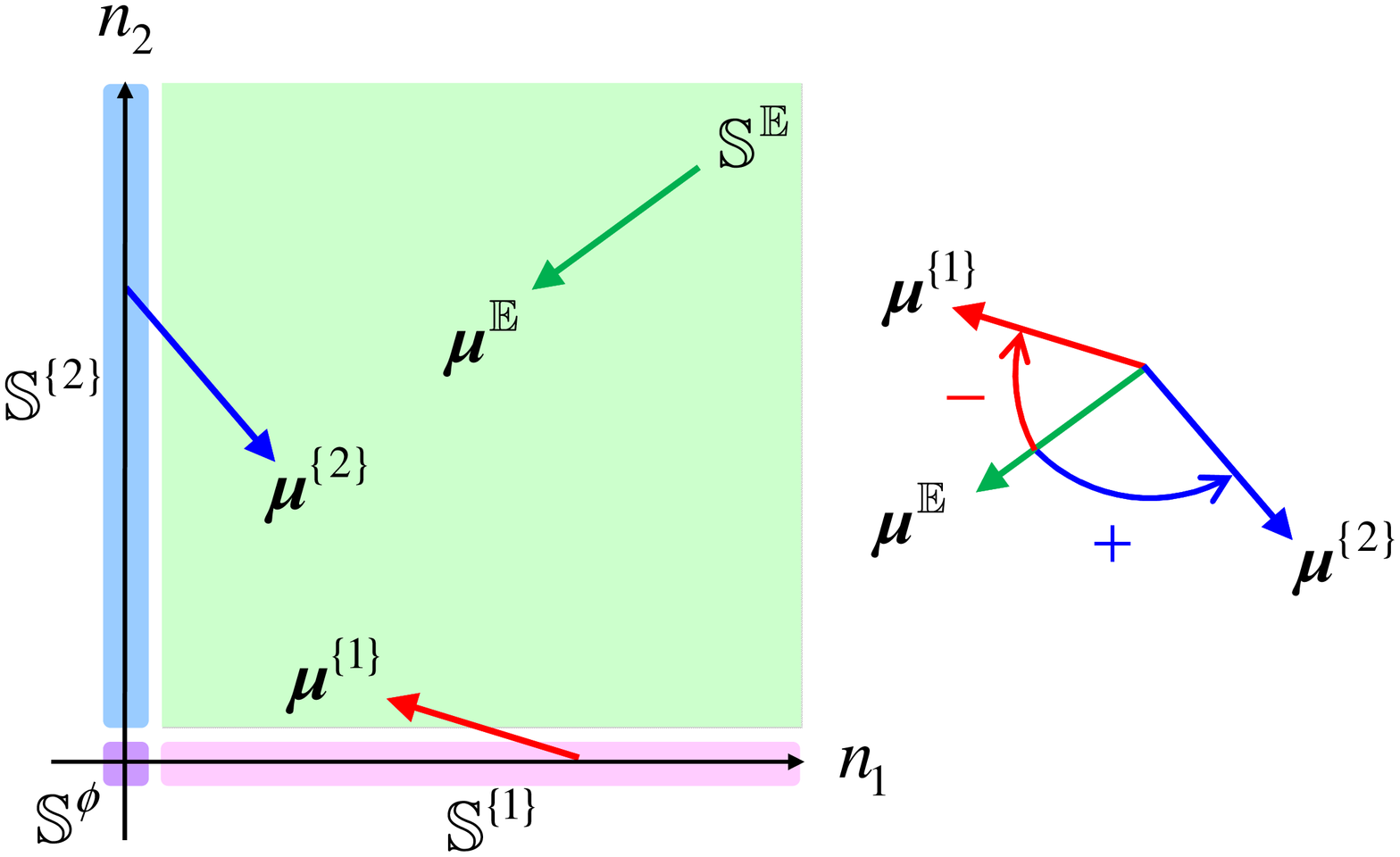}
\label{fig:stability-a}
}
\hspace{0mm}
\subfigure[Case (b)]{
\includegraphics[scale=0.34]{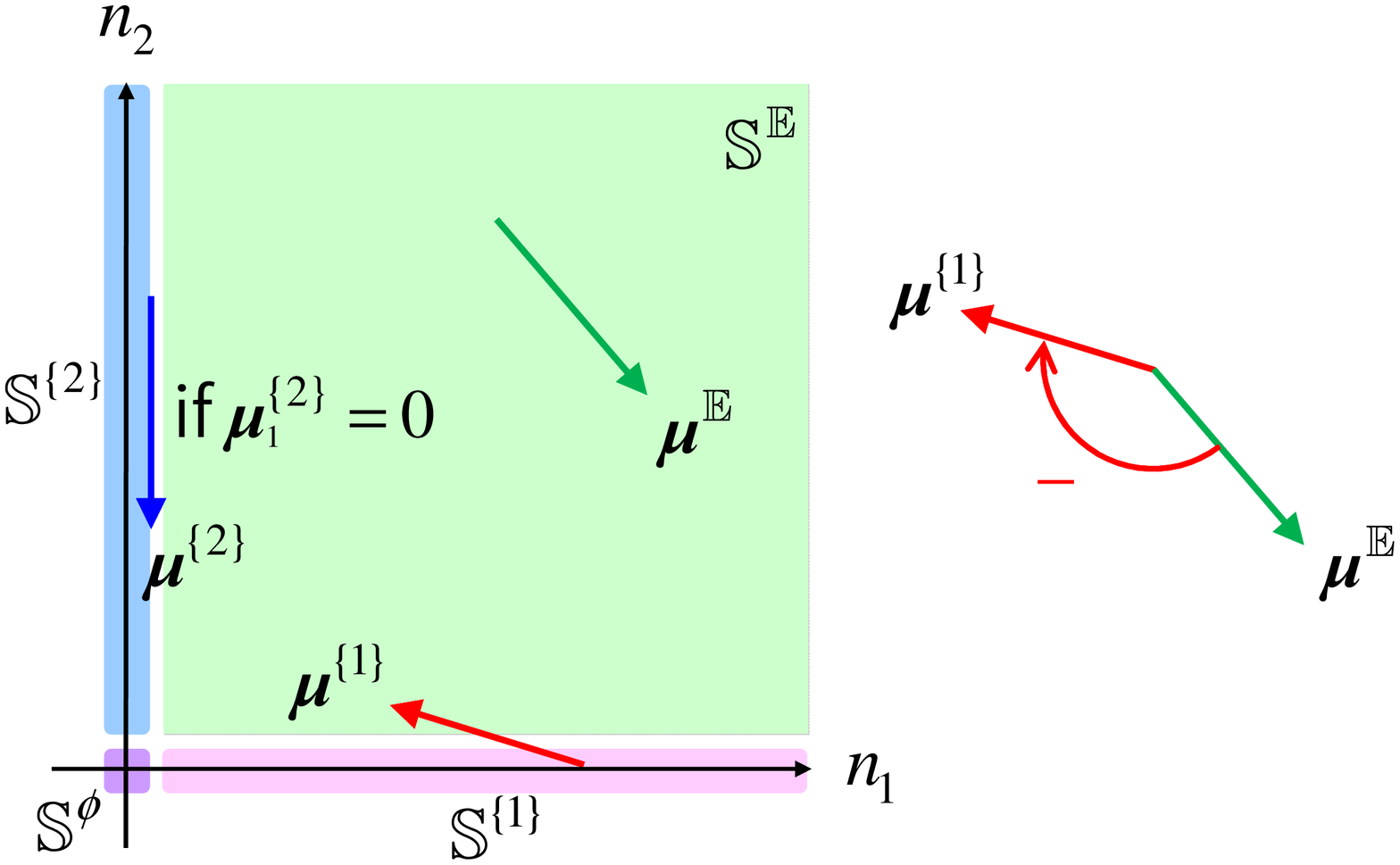}
\label{fig:stability-b}
}
\\
\subfigure[Case (c)]{
\includegraphics[scale=0.34]{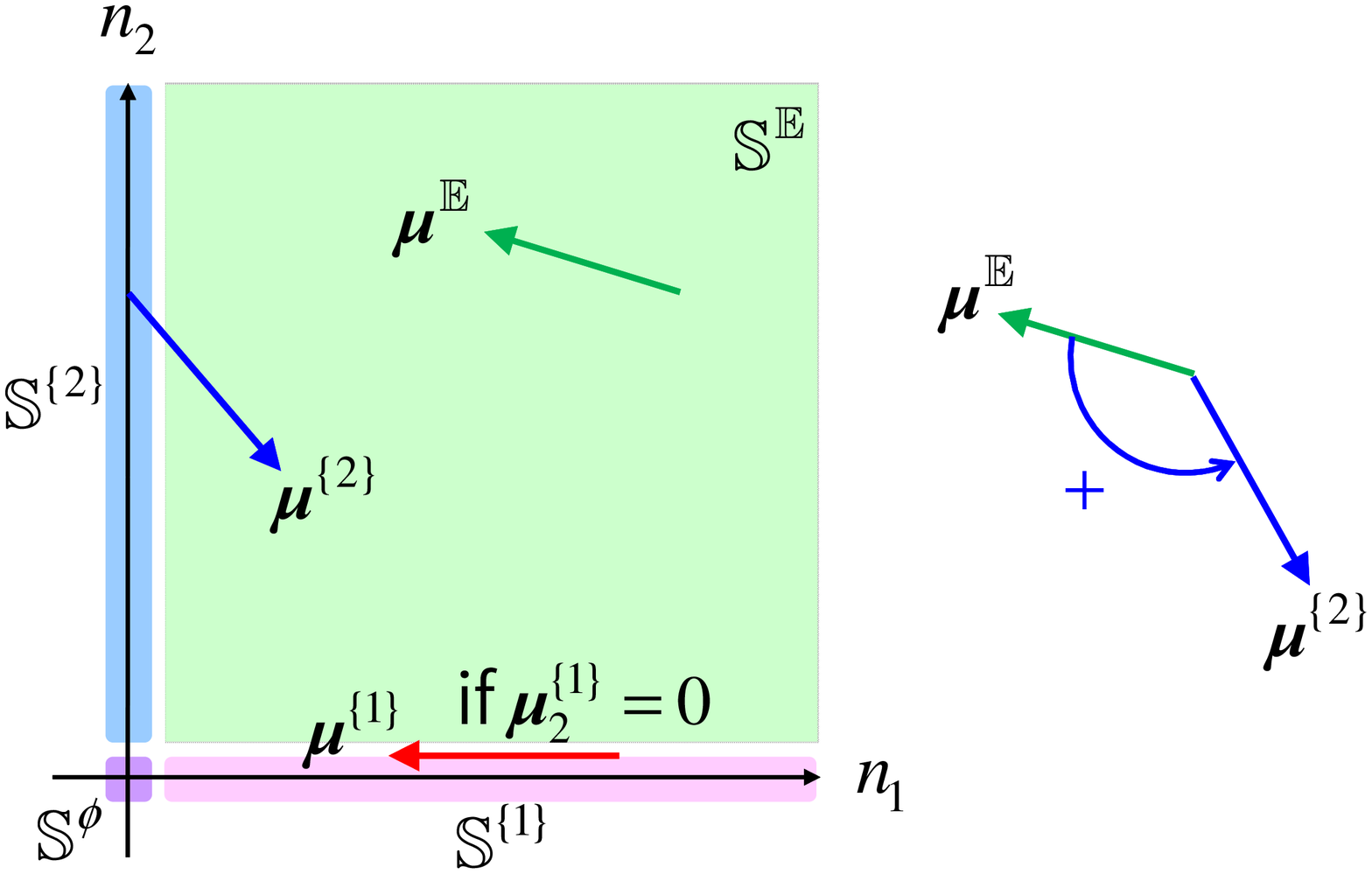}
\label{fig:stability-c}
}
\caption{Stability condition}
\end{figure}
\begin{rem}
Provided $\vc{\mu}^{\bbE} \neq \vc{0}$, Assumption~\ref{ass:stability}
holds if and only if the 2D-RRW $\{\vc{Z}(\ell)\}$ has the unique
stationary distribution. It is known that, even if $\vc{\mu}^{\bbE} =
\vc{0}$, the 2D-RRW $\{\vc{Z}(\ell)\}$ can be stable though its
stationary distribution must be heavy-tailed. For details, see
\cite{Fayo95, Koba14}.
\end{rem}

\subsection{Moment generating functions of increments}

In this subsection, we discuss the moment generating functions of the
increments $\vc{X}^{\bbA}$'s and describe the technical conditions used to
obtain the main results of this paper.

Let $\gamma^{\bbA}$, $\bbA \in 2_+^{\bbE}$, denote the moment
generating function of $\vc{X}^{\bbA}$, i.e.,
\begin{align}
\gamma^{\bbA}(\vc{\theta}) 
= \EE[\exp \{ \br{\vc{\theta},\vc{X}^{\bbA}} \}], \quad 
\vc{\theta} 
:= (\theta_{1},\theta_{2}) \in \bbR^{2}, 
\label{defn-gamma(theta)}
\end{align}
where $\br{\vc{x},\vc{y}}$ denotes the inner product of vectors
$\vc{x}$ and $\vc{y}$. From (\ref{defn-mu}) and
(\ref{defn-gamma(theta)}), we have
\begin{align}
\mu_{a}^{\bbA}
= \EE[X_{a}^{\bbA} ]
=
{\partial\gamma^{\bbA} \over \partial \theta_{a}} (\vc{0}), \qquad 
\bbA \in 2_+^{\bbE},\ a \in \{1,2\}. 
\label{eqn-mu_a^A}
\end{align}
Furthermore,  let
$\mathsf{\mathsf{\Gamma}}^{\bbA}$ and $\partial \mathsf{\Gamma}^{\bbA}
$, $\bbA \in 2_+^{\bbE}$, denote
\begin{align}
\mathsf{\Gamma}^{\bbA} 
= \{\vc{\theta}; \gamma^{\bbA}(\vc{\theta}) < 1\}, 
\quad 
\partial \mathsf{\Gamma}^{\bbA} 
= \{\vc{\theta}; \gamma^{\bbA}(\vc{\theta}) = 1\},
\label{eqn:gamma<1}
\end{align}
respectively. 

Under Assumption~\ref{ass:stability}, we have the following propositions (see \cite[Remark~2 and Lemma~2]{Koba14}).
\begin{prop}
\label{pro:property-Gamma}
If Assumption~\ref{ass:stability} holds, then the following are true:
\begin{enumerate}
\renewcommand{\labelenumi}{(\Roman{enumi})}
\item For each $\bbA \in 2_+^{\bbE}$, $\mathsf{\Gamma}^{\bbA}$ is a
  convex set and $\vc{0} \in \partial \mathsf{\Gamma}^{\bbA}$.
\item $\mathsf{\Gamma}^{\bbE} \cap \bbR_{>0}^2 \neq \emptyset $, where
  $\bbR_{>0} = \{x \in \mathbb{R}: x > 0 \} = (0,\infty)$.
\end{enumerate}
\end{prop}
\begin{prop}
\label{pro:lemma_2}
Assumption~\ref{ass:stability} holds if and only if, for each $a \in
\{1,2\}$, there exists some $\vc{\theta} \in \partial
\mathsf{\Gamma}^{\mathbb{E}} \cap \partial \mathsf{\Gamma}^{\{a\}}$
such that $\theta_{a} > 0$.
\end{prop}

\begin{rem}
The symbol ``$\partial$" is dropped in the original statement of \cite[Lemma 2]{Koba14}. 
\end{rem}

We now present a geometric property of $\mathsf{\Gamma}^{\{1\}}$ and
$\mathsf{\Gamma}^{\{2\}}$.
\begin{lem}
\label{lem:geometric-property}
Suppose that Assumption~\ref{ass:stability} holds. For each $a \in
\{1,2\}$, $\mu_a^{\{a\}} < 0$ if and only if $\mathsf{\Gamma}^{\{a\}}
\cap \bbR_{>0}^2 \neq \emptyset$.
\end{lem}

\begin{proof}
We fix $a \in \{1,2\}$ arbitrarily. Since $|X_1^{\{a\}}| =
|X_2^{\{a\}}| \le 1$, the moment generating function
$\gamma^{\{a\}}(\vc{\theta})$ is holomorphic at $\vc{\theta} =
\vc{0}$. It thus follows from the Taylor's expansion of
$\gamma^{\{a\}}(\vc{\theta})$ at $\vc{\theta} = \vc{0}$ that the
following holds at a neighborhood of $\vc{\theta}=\vc{0}$:
\begin{eqnarray}
\gamma^{\{a\}}(\vc{\theta})
= \gamma^{\{a\}}(\vc{0}) 
+ \langle 
(\nabla_{\!\vc{\theta}} \gamma^{\{a\}})(\vc{0}),
\vc{\theta}
 \rangle
+ o(\| \vc{\theta} \|^2),
\label{Taylor-expansion-gamma}
\end{eqnarray}
where $\|\cdot\|$ denotes an arbitrary norm on $\bbR^2$, and where
$\nabla_{\!\vc{\theta}}$ denotes the gradient operator with respect to
a vector variable $\vc{\theta}$, i.e.,
\[
\nabla_{\!\vc{\theta}} 
= \left( 
{\partial \over \partial \theta_1},
{\partial \over \partial \theta_2} 
\right).
\]
Furthermore, it follows from
(\ref{eqn-mu_a^A}) that
\begin{equation*}
(\nabla_{\!\vc{\theta}} \gamma^{\{a\}})(\vc{0})
= \left(
{\partial\gamma^{\{a\}} \over \partial \theta_1}(\vc{0}),
{\partial\gamma^{\{a\}} \over \partial \theta_2}(\vc{0}),
\right)
= (\mu_1^{\{a\}}, \mu_2^{\{a\}}).
\end{equation*}
Substituting this equation and $\gamma^{\{a\}}(\vc{0}) =1$ into
(\ref{Taylor-expansion-gamma}) yields
\begin{eqnarray*}
\gamma^{\{a\}}(\vc{\theta}) - 1
= \mu_1^{\{a\}}\theta_1  + \mu_2^{\{a\}}\theta_2
+ o(\| \vc{\theta} \|^2),
\end{eqnarray*}
which implies that $\mathsf{\Gamma}^{\{a\}} \cap \bbR_{>0}^2 \neq
\emptyset$ if and only if $\mu_1^{\{a\}} \theta_1 + \mu_2^{\{a\}} \theta_2 < 0$ for some $\vc{\theta} > \vc{0}$, or equivalently, either $\mu_1^{\{a\}} < 0$ or $\mu_2^{\{a\}} < 0$. Recall here that $\mu_{3-a}^{\{a\}} \ge 0$ due to (\ref{ineqn-mu}). As a result, $\mu_a^{\{a\}} < 0$ if and only if
$\mathsf{\Gamma}^{\{a\}} \cap \bbR_{>0}^2 \neq \emptyset$.
\end{proof}

We now introduce the following assumption (see Fig.~\ref{fig:negative_drifts}).
\begin{assumpt}\label{assumpt-negative-drift}
$\mu_{1}^{\{1\}} < 0$, $\mu_{2}^{\{2\}} < 0$ and
\begin{equation}
\vc{\mu}^{\{1\}} \wedge \vc{\mu}^{\{2\}}
= \mu_{1}^{\{1\}} \mu_{2}^{\{2\}} - \mu_{2}^{\{1\}} \mu_{1}^{\{2\}} > 0. 
\label{cond-mu^{1}_X_mu^{2}>0}
\end{equation}
\end{assumpt}
\begin{figure}[htbp]
	\centering
	\includegraphics[scale=0.35]{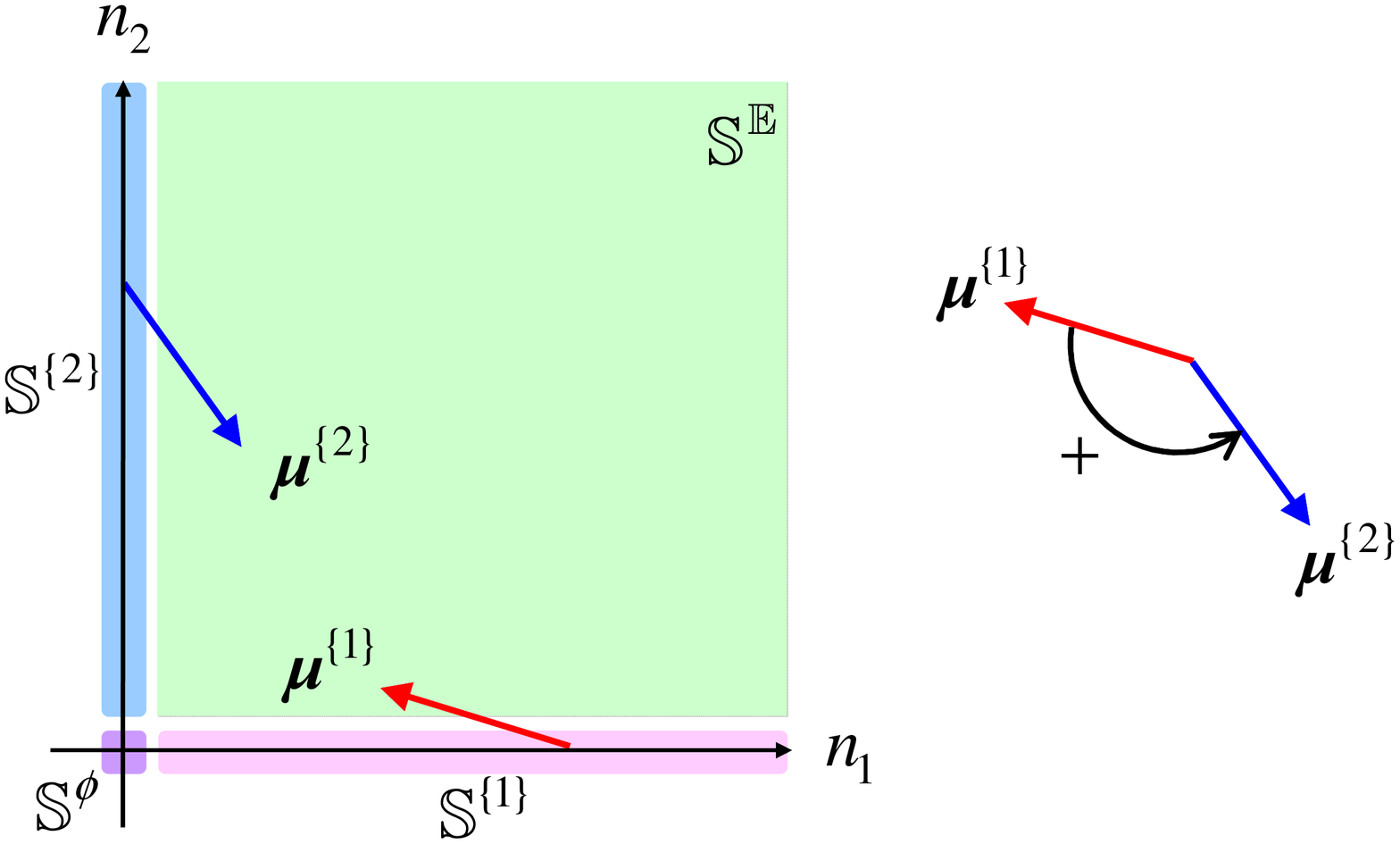}
	\caption{Mean drifts satisfying Assumption~\ref{assumpt-negative-drift}}
	\label{fig:negative_drifts}
\end{figure}

\begin{rem}
Recall that $\mu_2^{\{1\}} \ge 0$ and $\mu_1^{\{2\}} \ge 0$, as shown
in (\ref{ineqn-mu}). If either $\mu_2^{\{1\}} = 0$ or $\mu_1^{\{2\}} =
0$ holds, then (\ref{cond-mu^{1}_X_mu^{2}>0}) follows from
$\mu_{1}^{\{1\}} < 0$ and $\mu_{2}^{\{2\}} < 0$.
\end{rem}

\begin{rem}
\label{rem:Jackson}
A typical example of 2D-RRWs satisfying Assumption~\ref{assumpt-negative-drift} is a two-node Jackson network with
cooperative servers. For details, see Appendix~\ref{appendix-Jackson}.
\end{rem}

From \lemt{geometric-property}, we have the
following result, which plays an important role in the next section.

\begin{lem}\label{pro:geometric-property-2}
Assumption~\ref{assumpt-negative-drift} holds if and only if
\begin{equation}
\bigcap_{\bbA \in 2_+^{\bbE}}\mathsf{\Gamma}^{\bbA} \cap \mathbb{R}^{2}_{>0}
= \mathsf{\Gamma}^{\{1\}} \cap \mathsf{\Gamma}^{\{2\}} \cap
\mathsf{\Gamma}^{\bbE} \cap \mathbb{R}^{2}_{>0}
\neq \emptyset,
\label{eqn-common-domain}
\end{equation}
provided that Assumption~\ref{ass:stability} is satisfied.
\end{lem}

\begin{proof}
\lemt{geometric-property} shows that
Assumption~\ref{assumpt-negative-drift} is necessary for that
$\bigcap_{\bbA \in 2_+^{\bbE}}\mathsf{\Gamma}^{\bbA} \cap \mathbb{R}^{2}_{>0} \neq \emptyset$. It
follows from (\ref{defn-gamma(theta)}) and (\ref{eqn:gamma<1}) that
the domains $\mathsf{\Gamma}^{\bbA} \cup \partial
\mathsf{\Gamma}^{\bbA}$'s, $\bbA \in 2_+^{\bbE}$, are convex and the
curves $\partial\mathsf{\Gamma}^{\bbA}$'s, $\bbA \in 2_+^{\bbE}$,
include the origin $(0,0)$.  Furthermore, $\vc{\mu}^{\{1\}}$ and
$\vc{\mu}^{\{2\}}$ are the gradient vectors, at the origin $(0,0)$, of
the functions $\gamma^{\{1\}}$ and $\gamma^{\{2\}}$, respectively and
they satisfy (\ref{ineqn-mu}). These facts, together with
Fig.~\ref{fig:geo_int_1}, imply that if $\vc{\mu}^{\{1\}} \wedge
\vc{\mu}^{\{2\}} > 0$ then $\bigcap_{\bbA \in 2_+^{\bbE}}\mathsf{\Gamma}^{\bbA} \cap \mathbb{R}^{2}_{>0} \cap
\mathbb{R}^{2}_{>0} \neq \emptyset$, and vice versa.
As a result, Assumption~\ref{assumpt-negative-drift} is equivalent to
(\ref{eqn-common-domain}), under Assumption~\ref{ass:stability}.
\end{proof}

\begin{figure}[htbp]
	\centering
	\includegraphics[scale=0.35]{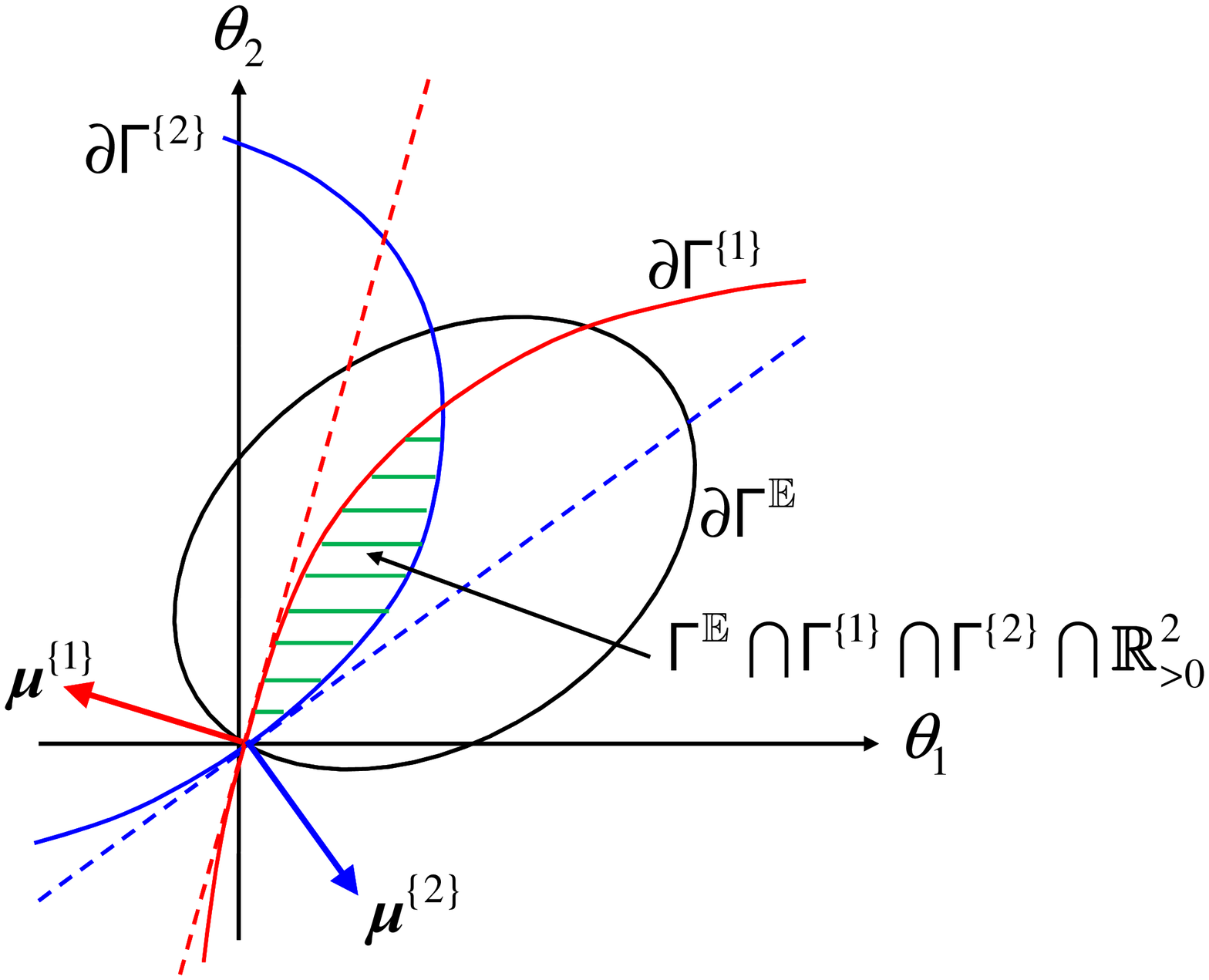}
	\caption{Relation between the domains $\mathsf{\Gamma}^{\bbE}$, $\mathsf{\Gamma}^{\{1\}}$ and $\mathsf{\Gamma}^{\{2\}}$}
	\label{fig:geo_int_1}
\end{figure}

\subsection{Geometric drift condition}

In this subsection, we show that a Foster-Lyapunov condition (Condition~\ref{cond:assumpt-f-ergodic} below) holds for the 
transition probability matrix $\vc{P}:=(p(\vc{n};\vc{m}))_{(\vc{n};\vc{m})\in\bbS^2}$ of the 2D-RRW $\{\vc{Z}(\ell)\}$ (see (\ref{defn-p(n;m)})), 
under Assumptions~\ref{ass:stability} and \ref{assumpt-negative-drift}.
\begin{cond}[{}{\cite[Section~14.2.1]{Meyn09}}]
\label{cond:assumpt-f-ergodic}
There exist a finite set $\bbK \subset \bbS$, $b \in \bbR_{>0}$, $c \in (0,1)$ and column vector $\vc{v}:=(v(\vc{n}))_{\vc{n} \in \bbS} \ge \vc{e}/c$
such that
\begin{equation}
\vc{P}\vc{v} - \vc{v} \le  - c\vc{v} + b \vc{1}_{\bbK},
\label{eqn:ineqn-Qv}
\end{equation}
where, for any set $\bbC \subseteq \bbS$,
$\vc{1}_{\bbC}:=(1_{\bbC}(\vc{n}))_{\vc{n}\in\bbS}$ denotes a column
vector such that
\[
1_{\bbC}(\vc{n})
=\left\{
\begin{array}{ll}
1, & \vc{n} \in \bbC,
\\
0, & \vc{n} \in \bbS\setminus\bbC.
\end{array}
\right.
\]
\end{cond}

It is known (\cite[Theorem~15.0.1]{Meyn09}) that if $\vc{P}$ is
irreducible and aperiodic and Condition~\ref{cond:assumpt-f-ergodic}
holds then $\vc{P}$ is geometrically ergodic. Thus, we refer to
Condition~\ref{cond:assumpt-f-ergodic} as the {\it geometric drift
  condition}.

Under Assumptions~\ref{ass:stability} and
\ref{assumpt-negative-drift}, we establish the geometric drift condition
on the 2D-RRW.
\begin{lem}\label{lem:v-02}
Suppose that Assumptions~\ref{ass:stability} and
\ref{assumpt-negative-drift} are satisfied. Fix $\vc{\theta} \in \bigcap_{\bbA \in 2_+^{\bbE}}\mathsf{\Gamma}^{\bbA} \cap \bbR_{>0}^2$ and $\vc{v}=(v(\vc{n}))_{\vc{n} \in \bbS}$ such that
\begin{align}
v(\vc{n}) = c^{-1}\exp\{\br{\vc{\theta},\vc{n}}\}, \qquad \vc{n} \in \bbS,
\label{defn-v}
\end{align} 
where 
\begin{equation}
c = 1 - 
\max(\gamma^{\{1\}}(\vc{\theta}),\gamma^{\{2\}}(\vc{\theta}),\gamma^{\bbE}(\vc{\theta})).
\label{eqn:defn-c}
\end{equation}
Under these conditions, Condition~\ref{cond:assumpt-f-ergodic} holds
for $\bbK = \{(0,0)\}$ and
\begin{equation}
b = 1 + c^{-1} (\gamma^{\emptyset}(\vc{\theta}) - 1) > 0.
\label{defn-b-02}
\end{equation}
\end{lem}

\proof It follows from (\ref{defn-gamma(theta)}) and (\ref{defn-v})
that, for $\vc{n} \in \bbS^{\bbA}$ and $\bbA \in 2_+^{\bbE}$,
\begin{eqnarray*}
\sum_{\vc{m} \in \bbU^{\bbA}} p^{\bbA}(\vc{m})v(\vc{n}+\vc{m})
&=& 
c^{-1}\sum_{\vc{m} \in \bbU^{\bbA}} p^{\bbA}(\vc{m})
\exp\{ \br{\vc{\theta},\vc{n} + \vc{m}} \}
\nonumber
\\
&=& c^{-1}\exp\{ \br{\vc{\theta},\vc{n} } \}
\sum_{\vc{m} \in \bbU^{\bbA}} 
p^{\bbA}(\vc{m})\exp\{ \br{\vc{\theta},\vc{m}} \}
\nonumber
\\
&=& v(\vc{n})\gamma^{\bbA}(\vc{\theta}) \le (1 - c) v(\vc{n}),
\end{eqnarray*}
where the inequality is due to (\ref{eqn:defn-c}).
Similarly, we have
\begin{eqnarray*}
\sum_{\vc{m} \in \bbU^{\emptyset}} p^{\emptyset}(\vc{m})v(\vc{m})
&=& c^{-1} \sum_{\vc{m} \in \bbU^{\emptyset}} 
p^{\emptyset}(\vc{m}) \exp\{ \br{ \vc{\theta}, \vc{m} } \} 
\nonumber
\\
&=& c^{-1} \gamma^{\emptyset}(\vc{\theta}) 
= (c^{-1} - 1) + b
\nonumber
\\
&=& (1-c)v(\vc{0}) + b,
\end{eqnarray*}
where the last two equalities are due to (\ref{defn-b-02}) and
$v(\vc{0}) = c^{-1}$.  Note here that, since $\vc{\theta} > \vc{0}$
and $\vc{X}^{\emptyset} \ge \vc{0}$, we have
$\gamma^{\emptyset}(\vc{\theta}) -1 \ge 0$ and thus $b >
0$. Therefore, (\ref{eqn:ineqn-Qv}) holds for $\mathbb{K} = \{(0,0)\}$
and $b$ given in (\ref{defn-b-02}). \qed

\section{QBD approximation}

This section describes the QBD approximation of the 2D-RRW
$\{\vc{Z}(\ell);\ell\in\bbZ_+\}$. We first reformulate the
2D-RRW $\{\vc{Z}(\ell)\}$ as a quasi-birth-and-death process (QBD)
such that its phase process is an infinite
birth-and-death process. We then provide the definition of the QBD
approximation of the 2D-RRW.

\subsection{Reformulation as double QBD}

For convenience, we refer to $Z_1(\ell)$ and $Z_2(\ell)$ as the {\it
  level} and the {\it phase}, respectively.  We then arrange the
elements of the state space $\bbS$ in a lexicographical order such
that the level is the primary variable and the phase is the secondary
one. Since the marginal processes $\{Z_1(\ell)\}$ and $\{Z_2(\ell)\}$
increases or decreases by at most one, the transition probability
matrix $\vc{P}$ of $\{\vc{Z}(\ell)\}$ has the QBD-structure:
\begin{align}
\vc{P}
= 
\left(
\begin{array}{@{}c@{~}c@{~}c@{~}c@{~}c@{~}c@{}}
\vc{B}(0) 	& 
\vc{B}(1) 	&
\vc{O} 		&
\vc{O} 		&
\vc{O} 		&
\cdots
\\
\vc{A}(-1) 	& 
\vc{A}(0) 	&
\vc{A}(1)	&
\vc{O} 		&
\vc{O} 		&
\cdots
\\
\vc{O}&
\vc{A}(-1) 	& 
\vc{A}(0) 	&
\vc{A}(1)	&
\vc{O} 		&
\cdots
\\
\vc{O}		&
\vc{O}		&
\vc{A}(-1) 	& 
\vc{A}(0) 	&
\vc{A}(1) 	&
\cdots
\\
\vdots		&
\vdots		&
\vdots		&
\vdots		&
\ddots		&
\ddots
\end{array}
\right),
\label{defn-P}
\end{align}
where $\vc{O}$ denotes the zero matrix, and where $\vc{A}(k)$, $k \in
\{0,\pm1\}$, and $\vc{B}(k)$, $k \in \{0,1\}$, are given by
\begin{align}
\vc{A}(k) 
&= 
\left(
\begin{array}{@{}c@{~\,}c@{~\,}c@{~\,}c@{~\,}c@{}}
p^{\{1\}}(k,0) 	& 
p^{\{1\}}(k,1) 	&
0				&
0				&
\cdots
\\
p^{\bbE}(k,-1) & 
p^{\bbE}(k,0) 	&
p^{\bbE}(k,1)  &
0				&
\cdots
\\
0				&
p^{\bbE}(k,-1) & 
p^{\bbE}(k,0) 	&
p^{\bbE}(k,1)  &
\cdots
\\
0				&
0				&
p^{\bbE}(k,-1) & 
p^{\bbE}(k,0) 	&
\cdots
\\
\vdots			&
\vdots			&
\vdots			&
\vdots			&
\ddots
\end{array}
\right),
\label{defn-A_k}
\end{align}
and
\begin{align}
\vc{B}(k) 
&= 
\left(
\begin{array}{@{}c@{~}c@{~}c@{~}c@{~}c@{}}
p^{\emptyset}(k,0) 	& 
p^{\emptyset}(k,1) 	&
0					&
0					&
\cdots
\\
\rule{0mm}{5mm} p^{\{2\}}(k,-1) & 
p^{\{2\}}(k,0) 		        	&
p^{\{2\}}(k,1)  	 	        &
0								&
\cdots
\\
\rule{0mm}{5mm} 0				&
p^{\{2\}}(k,-1) 				& 
p^{\{2\}}(k,0) 					&
p^{\{2\}}(k,1)  				&
\cdots
\\
\rule{0mm}{5mm} 0				&
0								&
p^{\{2\}}(k,-1) 				& 
p^{\{2\}}(k,0) 					&
\cdots
\\
\rule{0mm}{5mm}	\vdots 			&
\vdots							&
\vdots							&
\vdots							&
\ddots					
\end{array}
\right),
\label{defn-B_k}
\end{align}
respectively. By interchanging
$Z_1(\ell)$ and $Z_2(\ell)$, we can construct another
QBD-structured transition probability matrix. This is why the reformulated
$\{\vc{Z}(\ell)\}$ is sometimes called a {\it double QBD}
\cite{Koba13,Miya09}.

\subsection{Definition of QBD approximation}\label{subsec-single}

For $n \in \bbZ_+$, let $\sqsub{n}\vc{A}(k):=(\lsub{[n]}A_{i,j}(k))_{i,j\in\bbZ_n}$, $k \in \{0,\pm1\}$, and
$\sqsub{n}\vc{B}(k):=(\lsub{[n]}B_{i,j}(k))_{i,j\in\bbZ_n}$, $k \in \{0,1\}$, denote
\begin{equation}
\sqsub{n}\vc{A}(k)
= \left(
\begin{array}{@{}c@{}c@{}c@{~}|@{~}c@{}}

  & 
  & 
  &
0 
\\

 						&
\pasub{n-1}\vc{A}(k)	&
  						&
\vdots 					 			
\\

  & 
  & 
  &
0 			
\\

  & 
  & 
  &
p^{\bbE}(k,1) 
\\
\hline
0								&
\cdots							& 
0~~p^{\bbE}(k,-1)       				&
\rule{0mm}{5mm}\dm\sum_{i=0}^1 p^{\bbE}(k,i) 
\end{array}
\right),
\label{defn-(n)A_k}
\end{equation}
and
\begin{equation}
\hspace{-2mm}
\sqsub{n}\vc{B}(k)
= \left(
\begin{array}{@{}c@{}c@{}c@{~}|@{~}c@{}}
  &
  &
  & 
0 
\\
  			&
\pasub{n-1}\vc{B}(k)	&
&
\vdots 			  		
\\
  & 
  & 
  &
0						
\\
  & 
  & 
  &
p^{\{2\}}\!(k,1) 
\\
\hline
0 					&
\cdots  			& 
0~~p^{\{2\}}\!(k,-1)       &
\rule{0mm}{5mm}\dm\sum_{i=0}^1 p^{\{2\}}\!(k,i) 
\end{array}
\right).
\label{defn-(n)B_k}
\end{equation}
respectively, where $\bbZ_n = \{0,1,\dots,n\}$, and where
$\lsub{(n)}\vc{A}(k):=(\lsub{(n)}A_{i,j}(k))_{i,j\in\bbZ_n}$ and
$\lsub{(n)}\vc{B}(k):=(\lsub{(n)}B_{i,j}(k))_{i,j\in\bbZ_n}$ are the
$(n+1) \times (n+1)$ northwest-corners of $\vc{A}(k)$ and $\vc{B}(k)$,
respectively. 
We then define
$\sqsub{n}\vc{P}:=(\sqsub{n}p(k,i;k',i'))_{(k,i;k',i')\in(\bbZ_+\times\bbZ_n)^2}$,
$n \in \bbN$, as
\begin{equation}
\sqsub{n}\vc{P}
= \left(
\begin{array}{@{}c@{\,}c@{\,}c@{\,}c@{\,}c@{}}
\sqsub{n}\vc{B}(0) & 
\sqsub{n}\vc{B}(1) &
\vc{O} &
\vc{O} &
\cdots
\\
\sqsub{n}\vc{A}(-1) 	& 
\sqsub{n}\vc{A}(0) 	&
\sqsub{n}\vc{A}(1) 	&
\vc{O} 			&
\cdots
\\
\vc{O}&
\sqsub{n}\vc{A}(-1) 	& 
\sqsub{n}\vc{A}(0) 	&
\sqsub{n}\vc{A}(1) 	&
\cdots
\\
\vc{O}			&
\vc{O}			&
\sqsub{n}\vc{A}(-1) 	& 
\sqsub{n}\vc{A}(0) 	&
\cdots
\\
\vdots			&
\vdots			&
\vdots			&
\vdots			&
\ddots
\end{array}
\right),
\label{defn-[n]P}
\end{equation}
where $(k,i;k',i')$ denotes ordered pair $((k,i),(k',i'))$ in $(\bbZ_+ \times \bbZ_n)^2$.
Clearly, $\sqsub{n}\vc{P}$ is the transition probability matrix of a
standard QBD, i.e., a QBD with infinite levels and finite
phases. Thus, we refer to $\sqsub{n}\vc{P}$ as the {\it QBD
  approximation} to $\vc{P}$.

We now consider a two-dimensional Markov chain
$\{\sqsub{n}\vc{Z}(\ell):=(\sqsub{n}Z_1(\ell),\sqsub{n}Z_2(\ell)); \ell\in\bbZ_+\}$ with state space $\bbS_n:= \bbZ_+ \times
\bbZ_n$ such that
\begin{align}
\sqsub{n}Z_1(\ell + 1)
&= \sqsub{n}Z_1(\ell) + \sum_{\bbA \in 2^{\bbE}} 
X_1^{\bbA}(\ell) I(\sqsub{n}\vc{Z}(\ell) \in \bbS^{\bbA}),
& \ell &\in \bbZ_+,
\label{defn-[n]Z_1(l)}
\\
\sqsub{n}Z_2(\ell + 1) 
&= \min\!\bigg(n,
\sqsub{n}Z_2(\ell) + \sum_{\bbA \in 2^{\bbE}} 
X_2^{\bbA}(\ell) I(\sqsub{n}\vc{Z}(\ell) \in \bbS^{\bbA}) 
\bigg), &  \ell &\in \bbZ_+.
\label{defn-[n]Z_2(l)}
\end{align}
It is easy to see that $\sqsub{n}\vc{P}$ is equal to the
transition probability matrix of the two-dimensional Markov chain
$\{\sqsub{n}\vc{Z}(\ell)\}$. Note here that
$\{\sqsub{n}\vc{Z}(\ell)\}$ is a 2D-RRW obtained by adding a
reflecting barrier at $n_2 = n$ to the state space $\bbS=\{(n_1,n_2)
\in \bbZ_+^2\}$ of the original 2D-RRW $\{\vc{Z}(\ell)\}$. Since
$\{\vc{Z}(\ell)\}$ is ergodic (i.e., irreducible, positive recurrent and aperiodic), it follows from (\ref{defn-[n]Z_1(l)})
and (\ref{defn-[n]Z_2(l)}) that $\{\sqsub{n}\vc{Z}(\ell)\}$ is
irreducible and positive recurrent. Therefore, $\sqsub{n}\vc{P}$ has
the unique stationary distribution, denoted by
$\sqsub{n}\vc{\pi}:=(\sqsub{n}\pi((k,i))_{(k,i)\in\bbS_n}$, which is
referred to as the {\it QBD approximation} to $\vc{\pi}$.

\section{Error bounds for QBD approximation}

In this section, we assume that Assumptions~\ref{ass:stability} and \ref{assumpt-negative-drift} hold, under which we present relative error bounds for the approximate time-averaged functional
obtained by the QBD approximation $\lsub{[n]}\vc{\pi}$.

For each $n \in \bbZ_+$, we extend $\lsub{[n]}\vc{\pi}$,  $\sqsub{n}\vc{A}(k)$ and $\sqsub{n}\vc{B}(k)$  to the respective orders of $\vc{\pi}$, $\vc{A}(k)$ and $\vc{B}(k)$ in such a way that
\begin{align}
&&&&&&&&
\lsub{[n]}\pi(k,i) &= 0, &
k &\in \bbZ_+, & i &\in \bbZ_+ \setminus \bbZ_n, &&&&&&&&
\label{eqn-[n]pi(k)=0}
\\
&&&&&&&&
\sqsub{n}A_{i,j}(k) &= 0,
& k &\in \{0,\pm 1\}, & (i,j) &\in \bbS \setminus \bbZ_n^2, &&&&&&&&
\label{eqn-[n]A(k)=0}
\\
&&&&&&&&
\sqsub{n}B_{i,j}(k) &= 0,
& k &\in \{0,1\}, & (i,j) &\in \bbS \setminus \bbZ_n^2. &&&&&&&&
\label{eqn-[n]B(k)=0}
\end{align}
Therefore, $\sqsub{n}\vc{P}$ is of the same
order as that of $\vc{P}$.
 
We now define $\vc{D}$ as the deviation matrix of $\vc{P}$ (see, e.g., \cite{Cool02}), i.e.,
\begin{equation}
\vc{D}
= \sum_{\ell = 0}^{\infty}
( \vc{P}^{\ell} - \vc{e}\vc{\pi} ).
\label{defn-D}
\end{equation}
Since $\vc{P}$ is geometrically ergodic (see Condition~\ref{cond:assumpt-f-ergodic} and Lemma~\ref{lem:v-02}), the deviation matrix $\vc{D}$
is well-defined (see, e.g., \cite[Theorem~15.0.1]{Meyn09}).  Furthermore, combining \cite[Section 4.1, Equation (9)]{Heid10} together with (\ref{eqn-[n]pi(k)=0})--(\ref{eqn-[n]B(k)=0}) yields
\[
\sqsub{n}\vc{\pi} - \vc{\pi}
= \sqsub{n}\vc{\pi} \left( \sqsub{n}\vc{P} - \vc{P} \right)\vc{D}, 
\qquad n  \in \bbN,
\]
which leads to
\begin{equation}
|\sqsub{n}\vc{\pi} - \vc{\pi}|
\le \sqsub{n}\vc{\pi} \left| \sqsub{n}\vc{P} - \vc{P} \right| |\vc{D}|, 
\qquad n  \in \bbN,
\label{eqn-diff-[n]pi}
\end{equation}
where $| \cdot |$ denotes the vector (resp.\ matrix) obtained by taking
the absolute values of the elements of the vector (resp. matrix) between
the vertical bars.  Therefore, we can estimate the absolute difference
between $\sqsub{n}\vc{\pi}$ and $\vc{\pi}$ once we obtain an upper
bound for $|\vc{D}|$.

For $\alpha \in (0,1)$, let
$\vc{S}^{(\alpha)}=(s^{(\alpha)}(k,i;k',i'))_{(k,i;k',i')\in\bbS^2}$
denote a stochastic matrix such that
\begin{equation}
\vc{S}^{(\alpha)}
= (1-\alpha) \sum_{\ell=0}^{\infty} \alpha^{\ell} \vc{P}^{\ell}
= (1-\alpha)(\vc{I} - \alpha\vc{P})^{-1} > \vc{O},
\label{defn-S^{(alpha)}}
\end{equation}
where $\vc{S}^{(\alpha)} > \vc{O}$ is due to the ergodicity of
$\{\vc{Z}(\ell)\}$. Furthermore, let
$\ol{\phi}_{\{\vc{0}\}}^{(\alpha)}$ denote
\begin{eqnarray}
\ol{\phi}_{\{\vc{0}\}}^{(\alpha)}
= 
\sup_{(k',i')\in\bbS}s^{(\alpha)}(0,0;k',i').
\label{defn-overline{varphi}_F_K}
\end{eqnarray}
It then follows from (\ref{defn-P}), (\ref{defn-S^{(alpha)}}) and
(\ref{defn-overline{varphi}_F_K}) that
\[
\ol{\phi}_{\{\vc{0}\}}^{(\alpha)}
\ge (1 - \alpha)[1 - \alpha p^{\{\emptyset\}}(0,0)]^{-1},
\]
and thus
\begin{equation}
\liminf_{\alpha \to 0}
\ol{\phi}_{\{\vc{0}\}}^{(\alpha)} \ge 1.
\label{liminf-ol{phi}}
\end{equation}
From (\ref{liminf-ol{phi}}) and Lemma~2.3 of \cite{Masu16-JORSJ}, we
obtain the following result.
\begin{lem}\label{lem-bound-h}
Suppose that all the conditions of Lemma~\ref{lem:v-02} are satisfied.
Let $\vc{g}:=(g(k,i))_{(k,i)\in\bbS}$ denote a nonnegative column
vector such that $\vc{0} \le \vc{g} \le c\vc{v}$. We then have
\begin{equation}
| \vc{D}|\, \vc{g}
\le (\vc{\pi} \vc{g} +1) 
\left(
\vc{v} 
+ {b \over c} \vc{e}
\right).
\label{bound-Dg}
\end{equation}
\end{lem}

\medskip

\begin{proof}
Let $\vc{Q} = \vc{P} - \vc{I}$. Clearly, the matrix $\vc{Q}$ is a
conservative and irreducible $q$-matrix (i.e., a diagonally dominant matrix with
negative diagonal elements and nonnegative off-diagonal elements
such that $\vc{Q}\vc{e} = \vc{0}$; see, e.g., \cite[Section
  1.2]{Ande91}). Thus, $\vc{Q}$ can be considered the infinitesimal
generator of a uniformizable continuous-time Markov
chain whose stationary distribution and deviation matrix are equal to
$\vc{\pi}$ and $\vc{D}$ (see \cite[Section 2]{Cool02}), respectively,
which results in
\begin{equation}
\vc{D} = \int_0^{\infty} 
\left( \exp\{\vc{Q} t\} - \vc{e}\vc{\pi} \right) \rmd t.
\label{eqn-D}
\end{equation}

We now fix $\beta = \alpha^{-1} - 1 > 0$. It then follows from
(\ref{defn-S^{(alpha)}}) and $\vc{Q} = \vc{P} - \vc{I}$ that
\begin{eqnarray}
\vc{S}^{(\alpha)} 
= (\vc{I} - \vc{Q}/\beta)^{-1} > \vc{O}.
\end{eqnarray}
Furthermore, pre-multiplying both sides of (\ref{eqn:ineqn-Qv}) by $\vc{\pi}$, we have 
\begin{equation}
\vc{\pi}\vc{v} \le b/c.
\label{ineqn-pi*v}
\end{equation}
Applying Lemma~2.3 of \cite{Masu16-JORSJ} to $\vc{D}$ in (\ref{eqn-D}) and using (\ref{ineqn-pi*v}), we readily obtain
\begin{equation}
| \vc{D}|\, \vc{g}
\le (\vc{\pi} \vc{g} +1) 
\left[
\vc{v} 
+ b\left\{
{1 \over c}
+ { 2 \over (\alpha^{-1} - 1)\ol{\phi}_{\{\vc{0}\}}^{(\alpha)} } 
\right\} \vc{e}
\right].
\label{bound-h}
\end{equation}
Finally, letting $\alpha \downarrow 0$ in (\ref{bound-h}) and
using (\ref{liminf-ol{phi}}) yields (\ref{bound-Dg}).
\end{proof}

\medskip

From Lemma~\ref{lem-bound-h}, we obtain the following result.
\begin{thm}\label{thm-f-ergodic}
If all the conditions of Lemma~\ref{lem:v-02} are satisfied, then
\begin{eqnarray}
\sup_{\vc{0} < \vc{g} \le c\vc{v}}
{ \left| \sqsub{n}\vc{\pi} - \vc{\pi} \right| \vc{g} \over \vc{\pi}\vc{g} }
\le E(n), \qquad n  \in \bbN,
\label{bound-E(n)}
\end{eqnarray}
where
\begin{eqnarray}
E(n)
&=& {12 \over c } \sum_{k=0}^{\infty}\lsub{[n]}\pi(k,n) \!
\left\{ \rme^{\theta_1+\theta_2} \rme^{k\theta_1+n\theta_2} 
+ b \right\}.
\label{defn-E(n)}
\end{eqnarray}
\end{thm}

\begin{rem}
The quantity $\vc{\pi}\vc{g}$ is equal to the following time-averaged
functional w.p.1 (see, e.g., \cite[Chapter 3,
  Proposition~4.1]{Brem99}):
\[
\vc{\pi}\vc{g}
= \lim_{N \to \infty}{1 \over N}\sum_{\ell = 1}^N g(Z_1(\ell),Z_2(\ell))\quad \mbox{w.p.1}.
\]
Furthermore, since $|(\sqsub{n}\vc{\pi} - \vc{\pi})\vc{g} | \le |
\sqsub{n}\vc{\pi} - \vc{\pi} |\vc{g}$ for $\vc{0} < \vc{g} \le
c\vc{v}$, the bound (\ref{bound-E(n)}) yields a relative error bounds
for the approximate time-averaged functional
$\sqsub{n}\vc{\pi}\vc{g}$:
\[
\sup_{\vc{0} < \vc{g} \le c\vc{v}}
{|(\sqsub{n}\vc{\pi} - \vc{\pi})\vc{g} | \over \vc{\pi}\vc{g} }
\le E(n), \qquad n  \in \bbN.
\]
\end{rem}

{\sc Proof of Theorem~\ref{thm-f-ergodic}.~}
Let 
\[
\vc{w} = (w(k,i))_{(k,i)\in\bbS} =
(\vc{w}(0)^{\top},\vc{w}(1)^{\top},\dots)^{\top},
\]
where $\vc{w}(k) :=(w(k,i))_{i\in\bbZ_+}$ is given by
\begin{equation}
\vc{w}(k)
= (\vc{\pi} \vc{g} +1) 
\left(
\vc{v}(k) 
+ {b \over c}\vc{e}
\right),
\qquad k \in \bbZ_+.
\label{defn-w(k)}
\end{equation}
It then follows from (\ref{bound-Dg}) that
\begin{equation}
|\vc{D}|\,\vc{g} \le \vc{w}.
\label{ineqn-|D|g}
\end{equation}
Post-multiplying both sides of (\ref{eqn-diff-[n]pi}) by $\vc{g}$ and
substituting (\ref{defn-P}), (\ref{defn-[n]P}) and (\ref{ineqn-|D|g})
into the resulting inequality yields
\begin{eqnarray}
\left| \lsub{[n]}\vc{\pi} - \vc{\pi} \right| \vc{g}
&\le& 
\lsub{[n]}\vc{\pi}  \left| \lsub{[n]}\vc{P} - \vc{P} \right| \vc{w}
\nonumber
\\
&=&
\lsub{[n]}\vc{\pi}(0)
\sum_{\nu=0}^{1} \left|\sqsub{n}\vc{B}(\nu) - \vc{B}(\nu) \right| \vc{w}(\nu)
\nonumber
\\
&& {} 
+
\sum_{k=1}^{\infty} \lsub{[n]}\vc{\pi}(k)
\sum_{\nu=-1}^{1} \left|\sqsub{n}\vc{A}(\nu) - \vc{A}(\nu) \right| 
\vc{w}(k+\nu). \qquad
\label{ineqn-(n)pi_m-pi-01}
\end{eqnarray}

In what follows, we estimate the right hand
side of (\ref{ineqn-(n)pi_m-pi-01}). From (\ref{defn-A_k}), (\ref{defn-(n)A_k}), and (\ref{eqn-[n]A(k)=0}) we have, for $\nu=0,\pm 1$,
\begin{eqnarray}
\sqsub{n}\vc{A}(\nu) - \vc{A}(\nu)
&=& \bordermatrix{
  		&
0 		&
\cdots 	&
 n - 1	&
 n	 	&
n + 1	&
n + 2	&
\cdots
\cr
0		&
0 		& 
\cdots 	& 
0  		&
0 & 
0 & 
0 & 
\cdots
\cr
\,\vdots	&
\vdots 		&
\ddots 		&
\vdots 		&
\vdots 		&
\vdots 		&
\vdots		&
\ddots 						
\cr
n-1			&
0 			& 
\cdots  	& 
0  			&
0 			&
0 			& 
0 			& 
\cdots
\cr
n			&
0 			&
\cdots  	& 
0    		&
 p^{\bbE}(\nu,1) &
-p^{\bbE}(\nu,1) &
0 				&
\cdots 
\cr
n+1		&
  \ast 	&
\cdots 	&
  \ast 	&
  \ast 	&
  \ast 	&
  \ast 	&
\cdots
\cr
n+2		&
  \ast 	&
\cdots 	&
  \ast 	&
  \ast 	&
  \ast 	&
  \ast 	&
\cdots
\cr
\,\vdots	&
\vdots 	&
\ddots 	&
\vdots 	&
\vdots 	&
\vdots 	&
\vdots 	&
\ddots
}. \qquad
\label{diff-(n)A_k}
\end{eqnarray}
From (\ref{defn-B_k}) and (\ref{defn-(n)B_k}) and (\ref{eqn-[n]B(k)=0}), we also have, for $\nu=0,1$,
\begin{eqnarray}
\sqsub{n}\vc{B}(\nu) - \vc{B}(\nu)
&=& \bordermatrix{
  		&
0 		&
\cdots 	&
n - 1	&
n	 	&
n + 1	&
n + 2	&
\cdots
\cr
0		&
0 		& 
\cdots 	& 
0  		&
0 & 
0 & 
0 & 
\cdots
\cr
\,\vdots		&
\vdots 			&
\ddots 		&
\vdots       		&
\vdots 		&
\vdots 		&
\vdots		&
\ddots 						
\cr
n-1			&
0 		& 
\cdots 	& 
0  		&
0 			&
0 			& 
0 			& 
\cdots
\cr
n			&
0 			&
\cdots  	& 
0    		&
 p^{\{2\}}(\nu,1) &
-p^{\{2\}}(\nu,1) &
0 				&
\cdots 
\cr
n+1		&
  \ast 	&
\cdots 	&
  \ast 	&
  \ast 	&
  \ast 	&
  \ast 	&
\cdots
\cr
n+2		&
  \ast 	&
\cdots 	&
  \ast 	&
  \ast 	&
  \ast 	&
  \ast 	&
\cdots
\cr
\,\vdots	&
\vdots 	&
\ddots 	&
\vdots 	&
\vdots 	&
\vdots 	&
\vdots 	&
\ddots
}.\qquad~~~
\label{diff-(n)B_k}
\end{eqnarray}
Using (\ref{eqn-[n]pi(k)=0}), (\ref{diff-(n)A_k}) and (\ref{diff-(n)B_k}), we obtain
\begin{eqnarray}
\lefteqn{
\lsub{[n]}\vc{\pi}(k)
\sum_{\nu=-1}^{1} |\sqsub{n}\vc{A}(\nu) - \vc{A}(\nu) |\, \vc{w}(k+\nu)
}
\qquad && 
\nonumber
\\
&=&  \lsub{[n]}\pi(k,n)
\sum_{\nu=-1}^{1} p^{\bbE}(\nu,1) \sum_{j=n}^{n+1} w(k+\nu,j),
\qquad k \in \bbN,
\label{eqn-(n)pi-A-h}
\\
\lefteqn{
\lsub{[n]}\vc{\pi}(0)
\sum_{\nu=0}^{1} |\sqsub{n}\vc{B}(\nu) - \vc{B}(\nu) |\, \vc{w}(\nu)
}
 \qquad && 
\nonumber
\\
&=& \lsub{[n]}\pi(0,n)
\sum_{\nu=0}^{1}  p^{\{2\}}(\nu,1) \sum_{j=n}^{n+1} w(\nu,j).
\label{eqn-(n)pi-B-h}
\end{eqnarray}
Applying (\ref{eqn-(n)pi-A-h}) and
(\ref{eqn-(n)pi-B-h}) to (\ref{ineqn-(n)pi_m-pi-01}) leads to
\begin{eqnarray*}
\left| \lsub{[n]}\vc{\pi} - \vc{\pi} \right| \vc{g}
&\le&
\sqsub{n}\pi(0,n)
\sum_{\nu=0}^{1}  p^{\{2\}}(\nu,1)\sum_{j=n}^{n+1} w(\nu,j)
\nonumber
\\
&& {} 
+ 
\sum_{k=1}^{\infty} \sqsub{n}\pi(k,n)
\sum_{\nu=-1}^{1} p^{\bbE}(\nu,1)\sum_{j=n}^{n+1} w(k+\nu,j).
\end{eqnarray*}
Furthermore, substituting (\ref{defn-w(k)}) into the above inequality
yields
\begin{eqnarray}
{ \left| \lsub{[n]}\vc{\pi} - \vc{\pi} \right| \vc{g} \over \vc{\pi}\vc{g}+1 }
&\le&
\lsub{[n]}\pi(0,n)
\sum_{\nu=0}^{1}  p^{\{2\}}(\nu,1)
\sum_{j=n}^{n+1} 
\left( 
v(\nu,j) + {b \over c}
\right) 
\nonumber
\\
&& {} 
+ 
\sum_{k=1}^{\infty} \lsub{[n]}\pi(k,n)
\sum_{\nu=-1}^{1} p^{\bbE}(\nu,1)
\sum_{j=n}^{n+1} 
\left(
v(k+\nu,j) + {b \over c} 
\right)
\nonumber
\\
&=&
\lsub{[n]}\pi(0,n) r^{\{2\}}(n)
+ 
\sum_{k=1}^{\infty} \lsub{[n]}\pi(k,n) r^{\bbE}(k,n),
\label{ineqn-|pi-[n]pi|g}
\end{eqnarray}
where 
\begin{align*}
r^{\{2\}}(n)
&= \sum_{\nu=0}^{1}  p^{\{2\}}(\nu,1)
\sum_{j=n}^{n+1} 
\left\{ 
v(\nu,j) + {b \over c} 
\right\},& n &\in \bbN,
\\
r^{\bbE}(k,n)
&= \sum_{\nu=-1}^{1} p^{\bbE}(\nu,1)
\sum_{j=n}^{n+1} 
\left\{
v(k+\nu,j) + {b \over c}
\right\}, & k,n & \in\bbN.
\end{align*}
Note here that $p^{\{2\}}(\nu,1) \le 1$ for $\nu\in\{0,1\}$,
$p^{\bbE}(\nu,1) \le 1$ for $\nu \in \{0,\pm1\}$ and $v(\vc{n}) \le
v(\vc{m})$ for $\vc{0} \le \vc{n} \le \vc{m}$. Thus, we have
\begin{align}
r^{\{2\}}(n)
&\le 4 \left\{ v(1,n+1) + {b \over c}\right\}, & n & \in \bbN,
\label{bound-r^{2}(n)}
\\
r^{\bbE}(k,n)
&\le 6 \left\{ v(k+1,n+1) + {b \over c} \right\}, & k,n & \in \bbN.
\label{bound-r^{D}(k,n)}
\end{align}
From (\ref{ineqn-|pi-[n]pi|g}), (\ref{bound-r^{2}(n)}) and
(\ref{bound-r^{D}(k,n)}), we obtain
\begin{eqnarray}
{ 
\left| \lsub{[n]}\vc{\pi} - \vc{\pi} \right| \vc{g} 
\over
\vc{\pi}\vc{g}+1 
} 
&\le& 4 \lsub{[n]}\pi(0,n) \! 
 \left\{ v(1,n+1) + {b \over c} \right\} 
\nonumber 
\\ 
&& {} + 6
\sum_{k=1}^{\infty}\lsub{[n]}\pi(k,n)\!  
\left\{ v(k+1,n+1) + {b \over   c} \right\} 
\nonumber 
\\ 
&\le& 6
\sum_{k=0}^{\infty}\lsub{[n]}\pi(k,n)\!  
\left\{ v(k+1,n+1) + {b \over   c} \right\} 
\nonumber 
\\ 
&=& {6 \over c}
\sum_{k=0}^{\infty}\lsub{[n]}\pi(k,n) \!  
\left\{
\rme^{\theta_1+\theta_2} \rme^{k\theta_1+n\theta_2} + b
\right\},\qquad~~
\label{add-eqn-160625-01}
\end{eqnarray}
we use (\ref{defn-v}) in the last equality.

We now note that
\begin{eqnarray}
\sup_{\vc{0} < \vc{g} \le c\vc{v}}
{ \left| \lsub{[n]}\vc{\pi} - \vc{\pi} \right| \vc{g} 
\over \vc{\pi}\vc{g} }
&=&
\sup_{
\scriptstyle 0 < \varepsilon \le 1 
\atop 
\scriptstyle \varepsilon \vc{e} \le \vc{g} \le \varepsilon (c \vc{v})
}
{ \left| \lsub{[n]}\vc{\pi} - \vc{\pi} \right| (\vc{g}/\varepsilon) 
\over \vc{\pi}(\vc{g}/\varepsilon) }
=
\sup_{
\vc{e} \le \vc{g} \le c\vc{v}
}
{ \left| \lsub{[n]}\vc{\pi} - \vc{\pi} \right| \vc{g} 
\over \vc{\pi}\vc{g} },
\label{add-eqn-160214-01}
\end{eqnarray}
and that
\begin{equation}
\sup_{
\vc{e} \le \vc{g} \le c\vc{v}
}
{ \vc{\pi}\vc{g} + 1 \over \vc{\pi}\vc{g} } \le 2.
\label{add-eqn-160327-01}
\end{equation}
Combining (\ref{add-eqn-160625-01}), (\ref{add-eqn-160214-01}) and
(\ref{add-eqn-160327-01}) results in (\ref{bound-E(n)}). \qed
\medskip

In the rest of this section, we simplify the bound (\ref{bound-E(n)}). To this end, we still
suppose that all the conditions of Lemma~\ref{lem:v-02} are satisfied. Since $\mathsf{\Gamma}^{\bbE}$,
$\mathsf{\Gamma}^{\{1\}}$ and $\mathsf{\Gamma}^{\{2\}}$ are open
sets, there exists $\wt{\vc{\theta}}:= (\wt{\theta}_1,\wt{\theta}_2) \in \bigcap_{\bbA \in 2_+^{\bbE}} \mathsf{\Gamma}^{\bbA} \cap \bbR_{>0}^2$ such that $\wt{\vc{\theta}} > \vc{\theta}$, where $\vc{\theta}$ is the vector appearing in Lemma~\ref{lem:v-02}. Using such a vector $\wt{\vc{\theta}}$, we can obtain a weaker but
simpler error bound.
\begin{thm}
Suppose that all the conditions of Lemma~\ref{lem:v-02} are
satisfied. Furthermore, fix $\wt{\vc{\theta}} \in \bigcap_{\bbA \in 2_+^{\bbE}} \mathsf{\Gamma}^{\bbA} \cap \bbR_{>0}^2$ such that
$\wt{\vc{\theta}} > \vc{\theta}$. We then have
\begin{eqnarray}
\sup_{\vc{0} < \vc{g} \le c\vc{v}}
{ \left| \sqsub{n}\vc{\pi} - \vc{\pi} \right| \vc{g} \over \vc{\pi}\vc{g} }
\le \wt{E}(n), \qquad n  \in \bbN,
\label{bound-wt{E}(n)}
\end{eqnarray}
where
\begin{eqnarray}
\wt{E}(n)
&=& {12 \wt{b} \over c}
\left[
{
\rme^{\theta_1+\theta_2}  \rme^{-n(\wt{\theta}_2 - \theta_2)}
\over 
1 - \rme^{-(\wt{\theta}_1 - \theta_1)}
}
+ 
{
b\rme^{-n\wt{\theta}_2}
\over 
1 - \rme^{-\wt{\theta}_1}
}
\right], \qquad
\label{defn-wt{E}(n)}
\end{eqnarray}
with
\begin{eqnarray}
\wt{c} 
&=& 1 - 
\max(
\gamma^{\{1\}}(\vc{\wt{\theta}}),
\gamma^{\{2\}}(\vc{\wt{\theta}}),
\gamma^{\bbE}(\wt{\vc{\theta}})
),
\label{eqn:defn-wt{c}}
\\
\wt{b} &=& 1 + \wt{c}\,^{-1} ( \gamma^{\emptyset}(\wt{\vc{\theta}}) - 1 ).
\end{eqnarray}
\end{thm}

\begin{proof}
We prove that $E(n) \le \wt{E}(n)$ for $n \in \bbN$.  Let
$\wt{\vc{v}}:=(\wt{v}(\vc{n}))_{\vc{n} \in \bbS}$ denote
\begin{align}
\wt{v}(\vc{n}) = \wt{c}\,^{-1}\exp\{\br{\wt{\vc{\theta}},\vc{n}}\}, \qquad 
\vc{n} \in \bbS.
\label{defn-wt{v}}
\end{align} 
Proceeding as in the proof of Lemma~\ref{lem:v-02}, we can readily
prove that
\begin{equation}
\vc{P}\wt{\vc{v}} - \wt{\vc{v}}
\le - \wt{c} \wt{\vc{v}} + \wt{b} \vc{1}_{\{\vc{0}\}}.
\label{ineqn-Pwt{v}}
\end{equation}
Since $\wt{v}(\vc{n}) \le \wt{v}(\vc{m})$ for $\vc{0} \le \vc{n} \le
\vc{m}$, it follows from (\ref{defn-P}) and (\ref{defn-[n]P}) that
$\lsub{[n]}\vc{P}\wt{\vc{v}} \le \vc{P}\wt{\vc{v}}$. Using this
inequality and (\ref{ineqn-Pwt{v}}), we have
\begin{equation}
\lsub{[n]}\vc{P}\wt{\vc{v}} - \wt{\vc{v}}
\le - \wt{c} \wt{\vc{v}} + \wt{b} \vc{1}_{\{\vc{0}\}}.
\label{ineqn-[n]Pwt{v}}
\end{equation}
Pre-multiplying both sides of (\ref{ineqn-[n]Pwt{v}}) by
$\lsub{[n]}\vc{\pi}$ and using $\lsub{[n]}\vc{\pi}\lsub{[n]}\vc{P} =
\lsub{[n]}\vc{\pi}$ yields
\begin{eqnarray}
\lsub{[n]}\vc{\pi} \wt{\vc{v}} \le \wt{b}/\wt{c}.
\label{inreqn-[n]pi-wt{v}}
\end{eqnarray}
Combining (\ref{inreqn-[n]pi-wt{v}}) with (\ref{defn-wt{v}}),
we obtain
\begin{eqnarray}
\lsub{[n]}\pi(k,n) \le {\wt{b} \over \wt{c} \wt{v}(k,n)}
&=& \wt{b} \rme^{-k\wt{\theta}_1 - n\wt{\theta}_2}.
\label{ineqn-[n]pi-01}
\end{eqnarray}
Substituting (\ref{ineqn-[n]pi-01}) into
(\ref{defn-E(n)}) results in 
\begin{eqnarray*}
E(n) &\le& 
{12 \wt{b} \over c} 
\Bigg[ \rme^{\theta_1+\theta_2}
\sum_{k=0}^{\infty}
\rme^{-k(\wt{\theta}_1 - \theta_1) - n (\wt{\theta}_2 - \theta_2)}
 + b
\sum_{k=0}^{\infty}\rme^{-k\wt{\theta}_1 - n\wt{\theta}_2}
\Bigg]
\nonumber
\\
&=& {12 \wt{b} \over c}
\left[
{
\rme^{\theta_1+\theta_2}  \rme^{-n(\wt{\theta}_2 - \theta_2)}
\over 
1 - \rme^{-(\wt{\theta}_1 - \theta_1)}
}
+ 
{
b\rme^{-n\wt{\theta}_2}
\over 
1 - \rme^{-\wt{\theta}_1}
}
\right]
= \wt{E}(n),
\end{eqnarray*}
where the last equality holds due to (\ref{defn-wt{E}(n)}).
\end{proof}

\section{Conclusions}
 
In this paper, we have considered a two-dimensional reflecting random
walk. Under some technical conditions, we have derived simple relative
error bounds for the approximate time-averaged functional by the QBD
approximation. A typical example of 2D-RRWs satisfying the technical conditions is a two-node Jackson network with
cooperative servers. Since the technical conditions are somewhat
restrictive, we will try to remove them, as part of our future work.

\appendix

\section{Two-node Jackson network with cooperative servers}\label{appendix-Jackson}

We introduce a two-node Jackson network with cooperative servers, and
show that its two-dimensional queueing process is a 2D-RRW satisfying
Assumptions \assn{stability} and \ref{assumpt-negative-drift}.

We consider a queueing network with two nodes, numbered $1$ and
$2$. We also refer to the server at node $i \in \{1,2\}$ as server
$i$.  We then assume that customers arrive at node $i$ according to a
Poisson process with rate $\lambda_i$. The processing time required by
server $i$ to complete the service of a customer is distributed with
an exponential distribution having mean $1/\sigma_i$, which is
independent of all the other events. In addition, we assume that each
server helps the other one while its node has no jobs. More
specifically, while node $i$ is not empty and node $3-i$ is empty, the
customers at node $i$ are served by both servers and thus their
service times are independent and identically distributed with an
exponential distribution having mean $1/(\sigma_1 +
\sigma_2)$. Finally, we assume that, when each customer in node $i$
finishes its service, it goes to node $3-i$ with probability $q_i$ or
leaves the network with probability $1 - q_i$, where $0 < q_i < 1$.

It is easy to see that the two-node queueing network described above
is a Jackson network (see, e.g., \cite{Koba14}). We refer to this
Jackson network as {\it the two-node Jackson network with cooperative
  servers}.  It is also known (see, e.g., \cite{Koba14}) that the
two-node Jackson network with cooperative servers is stable if and
only if
\begin{align}
\label{eqn:stability_Jackson}
\rho_{1} 
:= \frac{\lambda_{1} + \lambda_{2} q_{2}}{\sigma_{1}(1 - q_{1}q_{2})} < 1, \quad 
\rho_{2} := \frac{\lambda_{2} + \lambda_{1} q_{1}}{\sigma_{2}(1 - q_{1}q_{2})} < 1.
\end{align}

We now assume, without loss of generality, that $\lambda_{1} + \lambda_{2}
+ \sigma_{1} + \sigma_{2} = 1$.
Using the uniformization technique (see, e.g.,
\cite[Section~4.5.2]{Tijm03}), the two-node Jackson network with
cooperative servers is formulated as a 2D-RRW, and its transition
probability laws $p^{\bbA}$'s, $\bbA \in 2^{\bbE}$, satisfy the
following equations:
\begin{align*}
p^{\mathbb{A}}(1,0) &= \lambda_1, \quad 
p^{\mathbb{A}}(0,1) = \lambda_2,  \qquad \quad\, 
\bbA \in 2^{\bbE},
\\
p^{\mathbb{A}}(1,1) &= p^{\mathbb{A}}(-1,-1) = 0, \qquad\qquad~~  
\bbA  \in 2^{\bbE},
\\
p^{\emptyset}(0,0) &= \sigma_1 + \sigma_2, \quad
p^{\mathbb{A}}(0,0) = 0, \quad~~   
\bbA  \in 2_+^{\bbE},
\\
p^{\mathbb{E}}(-1,1) 
&= \sigma_{1} q_{1}, \quad~~\, p^{\mathbb{E}}(1,-1) 
= \sigma_{2} q_{2}, 
\\
p^{\mathbb{E}}(-1,0) 
&= \sigma_{1}(1 - q_{1}), \quad~~~ 
p^{\mathbb{E}}(0,-1) 
= \sigma_{2}(1 -q_{2}),
\\
p^{\{1\}}(-1,1) 
&= (\sigma_{1} + \sigma_{2})q_{1}, 
~~
p^{\{2\}}(1,-1) 
= (\sigma_{1} + \sigma_{2})q_{2},
\\
p^{\{1\}}(-1,0) 
&= (\sigma_{1} + \sigma_{2})(1 - q_{1}),
\\ 
p^{\{2\}}(0,-1) 
&= (\sigma_{1} + \sigma_{2})(1 - q_{2}).
\end{align*}
From these transition probability laws, we obtain
\begin{align*}
&\mu_{1}^{\mathbb{E}} 
= \lambda_{1} 
+ \sigma_{2} q_{2} - \sigma_{1},
&
&\mu_{2}^{\mathbb{E}} 
= \lambda_{2} 
+ \sigma_{1} q_{1} - \sigma_{2},
\\
& \mu_{1}^{\{1\}} 
= \lambda_{1} - \sigma_{1} - \sigma_{2},
&
& \mu_{2}^{\{1\}} 
= \lambda_{2}  + (\sigma_{1} 
+ \sigma_{2})q_{1},
\\
& \mu_{1}^{\{2\}} 
= \lambda_{1}  + (\sigma_{1} 
+ \sigma_{2})q_{2},
&
& \mu_{2}^{\{2\}} 
= \lambda_{2} -  \sigma_{1} - \sigma_{2}.
\end{align*}
We can see that the condition \eqn{stability_Jackson} is satisfied if
and only if \asst{stability} holds. Furthermore, using
\eqn{stability_Jackson}, we can readily confirm that the present
2D-RRW satisfies Assumption \ref{assumpt-negative-drift}.



%
%
%

\end{document}